\documentclass[11pt]{article}%
\usepackage{amsmath}
\usepackage{amsfonts}
\usepackage{amssymb}
\usepackage{graphicx}%
\newtheorem{theorem}{Theorem}

\newtheorem{definition}[theorem]{Definition}

\newtheorem{proposition}[theorem]{Proposition}

\newenvironment{proof}[1][Proof]{\noindent\textbf{#1.} }{\ \rule{0.5em}{0.5em}}
\begin{document}

\title{Planar Harmonic and Monogenic Polynomials of Type \textit{A}}
\author{Charles F. Dunkl\\Dept. of Mathematics,\\University of Virginia,\\P.O. Box 400137, Charlottesville, VA 22904-4137}
\date{7 September 2016}
\maketitle

\begin{abstract}
Harmonic polynomials of type A are polynomials annihilated by the Dunkl
Laplacian associated to the symmetric group acting as a reflection group on
$\mathbb{R}^{N}$. The Dunkl operators are denoted by $T_{j}$ for $1\leq j\leq
N$, and the Laplacian $\Delta_{\kappa}=\sum_{j=1}^{N}T_{j}^{2}$. This paper
finds the homogeneous harmonic polynomials annihilated by all $T_{j}$ for
$j>2$. The structure constants with respect to the Gaussian and sphere inner
products are computed. These harmonic polynomials are used to produce
monogenic polynomials, those annihilated by a Dirac-type operator.

\end{abstract}
\maketitle

\section{Introduction}

The symmetric group $\mathcal{S}_{N}$ acts on $x\in\mathbb{R}^{N}$ as a
reflection group by permutation of coordinates. The group is generated by
reflections in the mirrors $\left\{  x:x_{i}=x_{j},i<j\right\}  $. The
function $w_{\kappa}\left(  x\right)  =\prod\limits_{1\leq i<j\leq
N}\left\vert x_{i}-x_{j}\right\vert ^{2\kappa}$ with parameter $\kappa$ is
invariant under this action and for $\kappa>-\frac{1}{N}$ there are several
measures that incorporate $w_{\kappa}$ and give rise to interesting
orthogonality structures. The corresponding measure on the $N$-torus is
related to the Calogero-Sutherland quantum-mechanical model of $N$ identical
particles on the circle with $1/r^{2}$ interaction potential, and the measure
$w_{\kappa}\left(  x\right)  e^{-\left\vert x\right\vert ^{2}/2}\mathrm{d}x$
is related to the model of $N$ identical particles on the line with $1/r^{2}$
interactions and harmonic confinement. This paper mainly concerns the measure
on the unit sphere in $\mathbb{R}^{N}$ for which there is an orthogonal
decomposition involving harmonic polynomials. In the present setting
\textit{harmonic} refers to the Laplacian operator $\Delta_{\kappa}$ produced
by the type-$A$ Dunkl operators.

For $x\in\mathbb{R}^{N}$ and $\left\{  i,j\right\}  \subset\left\{
1,2,\ldots,N\right\}  $ set $x\left(  i,j\right)  =\left(  \ldots,\overset
{i}{x_{j}},\ldots,\overset{j}{x_{i}},\ldots\right)  $, that is, entries $\#i$
and $\#j$ are interchanged.

\begin{definition}
For a polynomial $f$ and $1\leq i\leq N$%
\begin{align*}
T_{i}f\left(  x\right)    & :=\frac{\partial}{\partial x_{i}}f\left(
x\right)  +\kappa\sum_{j=1,j\neq i}^{N}\frac{f\left(  x\right)  -f\left(
x\left(  i,j\right)  \right)  }{x_{i}-x_{j}},\\
\Delta_{\kappa}f\left(  x\right)    & :=\sum_{i=1}^{N}T_{i}^{2}.
\end{align*}

\end{definition}

The (\textit{Dunkl}) operators $T_{i}$ mutually commute and map polynomials to
polynomials. The background for the theory can be found in the treatise
\cite[Ch. 6, Ch. 10.2]{DX}. An orthogonal basis for $L^{2}\left(
\mathbb{R}^{N},w_{\kappa}\left(  x\right)  e^{-\left\vert x\right\vert ^{2}%
/2}\mathrm{d}x\right)  $ can be defined in terms of products $f\left(
x\right)  L_{n}^{\lambda}\left(  \left\vert x\right\vert ^{2}/2\right)  $
where $f$ comes from an orthogonal set of harmonic homogeneous polynomials and
the Laguerre polynomial index $\lambda=\deg f-1+\frac{N}{2}\left(  \left(
N-1\right)  \kappa+1\right)  $. However attempts to explicitly construct
harmonic polynomials run into technical complications, presumably due to the
fact that the sign-changes (example: $x\mapsto\left(  -x_{1},x_{2}%
,\ldots,x_{N}\right)  $) are not elements of the symmetry group and thus the
$+/-$ symmetry of $T_{i}^{2}$ can not be used. To start on the construction
problem we will determine all the harmonic homogeneous polynomials annihilated
by $T_{j}$ for $2<j\leq N$. They are the analogues of ordinary harmonic
polynomials in two variables and thus we call them \textit{planar. }In this
situation there is a natural symmetry based on the transposition $\left(
1,2\right)  $: polynomials $f$ satisfying $f\left(  x\left(  1,2\right)
\right)  =f\left(  x\right)  $ are called \textit{symmetric} and those
satisfying $f\left(  x\left(  1,2\right)  \right)  =-f\left(  x\right)  $ are
called \textit{antisymmetric.} Then $T_{1}+T_{2}$ preserves the symmetry type
and $T_{1}-T_{2}$ reverses it. This property is relevant since $T_{1}%
^{2}+T_{2}^{2}=\frac{1}{2}\left(  T_{1}+T_{2}\right)  ^{2}+\frac{1}{2}\left(
T_{1}-T_{2}\right)  ^{2}$.

Section 2 describes the basis of polynomials used in the construction, sets up
and solves the recurrence equations required to produce symmetric and
antisymmetric harmonic polynomials. Also the formulae for the actions of
$T_{1}\pm T_{2}$ on the harmonics are derived. In Section 3 the inner product
structures involving the weight function $w_{\kappa}$ are defined and the
structural constants for the harmonic polynomials are computed. By means of
Clifford algebra techniques one can define an operator of Dirac type and
Section 4 describes this theory and produces the planar monogenic polynomials.
Finally Section 5 contains technical material providing proofs for some of the
results appearing in Sections 2 and 3.

\section{The \textit{p}-Basis and Construction of Harmonic Polynomials}

The natural numbers $\left\{  0,1,2,3\ldots\right\}  $ are denoted by
$\mathbb{N}_{0}$. The largest integer $\leq t\in\mathbb{R}$ is denoted by
$\left\lfloor t\right\rfloor $. Suppose $f$ is a polynomial in $x\in
\mathbb{R}^{N}$ then $\left(  1,2\right)  f$ denotes the polynomial $f\left(
x\left(  1,2\right)  \right)  $. To facilitate working with generating
functions we introduce the notation $\mathrm{coef}\left(  f,g_{j}\right)
:=c_{j}$ for the designated coefficient of $f$ in the expansion $f=\sum
_{i}c_{i}g_{i}$ in terms of a basis $\left\{  g_{i}\right\}  $. Throughout
$\kappa$ is a fixed parameter, implicit in $\left\{  T_{i}\right\}  $,
generally subject to $\kappa>-\frac{1}{N}$.

The $p$-basis associated to the operators $\left\{  T_{i}\right\}  $ is
constructed as follows: for $1\leq i\leq N$ the polynomials $p_{n}(x_{i};x)$
are given by the generating function
\[
\sum_{n=0}^{\infty}p_{n}(x_{i};x)r^{n}=(1-rx_{i})^{-1}\prod_{j=1}^{N}%
(1-rx_{j})^{-\kappa};
\]
then for $\alpha=(\alpha_{1},\ldots,\alpha_{N})\in\mathbb{N}_{0}^{N}$ (the
multi-indices), define $p_{\alpha}:=\prod_{i=1}^{N}p_{\alpha_{i}}(x_{i};x)$.
The set $\left\{  p_{\alpha}:\alpha\in\mathbb{N}_{0}^{N}\right\}  $ is a basis
for the polynomials for generic $\kappa$. The key property is that $T_{j}%
p_{n}(x_{i};x)=0$ for $j\neq i$. From \cite[Sect. 10.3]{DX} we find
\begin{gather}
T_{i}p_{\alpha}=(N\kappa+\alpha_{i})p_{\alpha_{i}-1}(x_{i};x)\prod_{m\neq
i}p_{\alpha_{m}}(x_{m};x)\label{TPA}\\
+\kappa\sum_{j\neq i}\sum_{m=0}^{\alpha_{j}-1}(p_{\alpha_{i}+\alpha_{j}%
-1-m}(x_{i};x)p_{m}(x_{j};x)-p_{m}(x_{i};x)p_{\alpha_{i}+\alpha_{j}-1-m}%
(x_{j};x))\nonumber\\
\times\prod_{n\neq i,j}p_{\alpha_{n}}(x_{n};x),\nonumber
\end{gather}
if $\alpha_{i}>0,$ and $T_{i}p_{\alpha}=0$ if $\alpha_{i}=0.$

Set up a symbolic calculus by letting $p_{j}^{n}$ denote $p_{n}\left(
x_{j};x\right)  $; more formally define a linear isomorphism from ordinary
polynomials to polynomials in the variables $\left\{  p_{1},\ldots
,p_{N}\right\}  $:%
\[
\Psi p_{\alpha}=p_{1}^{\alpha_{1}}p_{2}^{\alpha_{2}}\cdots p_{N}^{\alpha_{N}%
},~\alpha\in\mathbb{N}_{0}^{N},
\]
extended by linearity. Thus $\Psi\sum_{n=0}^{\infty}p_{n}(x_{i};x)r^{n}%
=\left(  1-p_{i}r\right)  ^{-1}$. In this form the action of $T_{i}$
(technically $\Psi T_{i}\Psi^{-1}$) on a function of $\left(  p_{1}%
,\ldots,p_{N}\right)  $ is given by%
\begin{align*}
T_{i}f\left(  p\right)   &  =\frac{\partial f}{\partial p_{i}}+N\kappa
\frac{f-\left(  p_{i}\rightarrow0\right)  f}{p_{i}}\\
&  +\kappa\sum_{j=1,j\neq i}^{N}\frac{\left(  p_{i}\rightarrow p_{j}\right)
f+\left(  p_{j}\rightarrow p_{i}\right)  f-f-\left(  p_{j}\longleftrightarrow
p_{i}\right)  f}{p_{i}-p_{j}}.
\end{align*}
The operators $\left(  p_{i}\rightarrow0\right)  $ and $\left(  p_{i}%
\rightarrow p_{j}\right)  $ replace $p_{i}$ by $0$ and $p_{j}$ respectively,
while $\left(  p_{j}\longleftrightarrow p_{i}\right)  $ is the transposition.
It suffices to examine the effect of the formula on monomials $p_{1}%
^{\alpha_{1}}p_{2}^{\alpha_{2}}\ldots p_{N}^{\alpha_{N}}$ and for $i=1.$ The
first two terms produce $(\alpha_{1}+N\kappa)$ \ if $\alpha_{1}>0,$ else $0.$
In the sum, the (typical) term for $j=2$ is $(p_{1}^{\alpha_{1}+\alpha_{2}%
}+p_{2}^{\alpha_{1}+\alpha_{2}}-p_{1}^{\alpha_{1}}p_{2}^{\alpha_{2}}%
-p_{1}^{\alpha_{2}}p_{2}^{\alpha_{1}})\prod_{m=3}^{N}p_{m}^{\alpha_{m}}%
/(p_{1}-p_{2}).$ A simple calculation shows this is the image under $\Psi$ of
the corresponding term in equation (\ref{TPA}) . This method was used in
\cite{D1} to find planar harmonics of type $B$ (the group generated by
sign-changes and permutation of coordinates).

From here on we will be concerned with polynomials in $p_{1},p_{2}$, that is,
exactly the set of polynomials annihilated by $T_{j}$ for $2<j\leq N$. Set
$p_{i,j}:=p_{i}\left(  x_{1};x\right)  p_{j}\left(  x_{2};x\right)  $ so that
$\Psi p_{i,j}=p_{1}^{i}p_{2}^{j}$. For each degree $\geq1$ there are two
independent harmonic polynomials, that is, $\left(  T_{1}^{2}+T_{2}%
^{2}\right)  f=0$, and a convenient orthogonal decomposition is by the action
of $\left(  1,2\right)  $; symmetric: $\left(  1,2\right)  f=f$, and
antisymmetric: $\left(  1,2\right)  f=-f$, to be designated by $+$ and $-$
superscripts, respectively. We use the operators $T_{1}+T_{2}$ and
$T_{1}-T_{2}$ (note $\left(  T_{1}^{2}+T_{2}^{2}\right)  =\frac{1}{2}\left(
T_{1}+T_{2}\right)  ^{2}+\frac{1}{2}\left(  T_{1}-T_{2}\right)  ^{2}$). The
harmonic polynomials will be expressed in the basis functions (symmetric)
$\phi_{nj}$ and (antisymmetric) $\psi_{nj}$ with generating functions
$u_{1},u_{2}$ (and $s:=\frac{1}{2}\left(  z+z^{-1}\right)  $) given by%
\begin{align*}
w_{1} &  :=(1-ztp_{1})^{-1}(1-z^{-1}tp_{2})^{-1},\\
w_{2} &  :=(1-z^{-1}tp_{1})^{-1}(1-ztp_{2})^{-1},\\
u_{1} &  :=\frac{1}{2}\left(  w_{1}+w_{2}\right)  =\frac{1-st(p_{1}%
+p_{2})+t^{2}p_{1}p_{2}}{(1-2stp_{1}+t^{2}p_{1}^{2})(1-2stp_{2}+t^{2}p_{2}%
^{2})},\\
u_{2} &  :=\left(  z-\frac{1}{z}\right)  ^{-1}\left(  w_{1}-w_{2}\right)
=\frac{t(p_{1}-p_{2})}{(1-2stp_{1}+t^{2}p_{1}^{2})(1-2stp_{2}+t^{2}p_{2}^{2}%
)},
\end{align*}%
\begin{align*}
u_{1} &  =\sum_{n=0}^{\infty}t^{n}\sum_{j=0}^{n}s^{j}\phi_{nj},\\
u_{2} &  =\sum_{n=1}^{\infty}t^{n}\sum_{j=0}^{n}s^{j}\psi_{nj}.
\end{align*}
There are parity conditions: $\phi_{nj}\neq0$ implies $j\equiv
n\operatorname{mod}2$ and $\psi_{nj}\neq0$ implies $j\equiv\left(  n-1\right)
\operatorname{mod}2$. These are formal power series and convergence is not
important (but is assured if $\max\left(  \left\vert zt\right\vert ,\left\vert
t/z\right\vert \right)  <\left(  \max_{i}\left\vert x_{i}\right\vert \right)
^{-1}$).

The following expressions are derived in Section 5: for $0\leq j\leq
\left\lfloor \frac{n}{2}\right\rfloor $%
\begin{align*}
\phi_{n,n-2j} &  =2^{n-1-2j}\sum_{i=0}^{j}\frac{\left(  n+1-2j\right)  _{2i}%
}{i!\left(  1-n+2j-2i\right)  _{i}}\left(  p_{n-j+i,j-i}+p_{j-i,n-j+i}\right)
\\
&  =2^{n-1-2j}\left(  p_{n-j,j}+p_{j,n-j}\right)  +2^{n-1-2j}\\
&  \times\sum_{i=1}^{j}\left(  n-2j+2i\right)  \frac{\left(  n+1-2j\right)
_{i-1}}{i!}\left(  -1\right)  ^{i}\left(  p_{n-j+i,j-i}+p_{j-i,n-j+i}\right)
;
\end{align*}
and for $0\leq j\leq\left\lfloor \frac{n-1}{2}\right\rfloor $%
\[
\psi_{n,n-1-2j}=2^{n-1-2j}\sum_{i=0}^{j}\frac{\left(  n-2j\right)  _{i}}%
{i!}\left(  -1\right)  ^{i}\left(  p_{n-j+i,j-i}-p_{j-i,n-j+i}\right)  .
\]

The reason for the use of this basis is that the actions of $T_{1}+T_{2}$ and
$T_{1}-T_{2}$ have relatively simple expressions. It is easy to verify that
(set $\partial_{v}:=\frac{\partial}{\partial v}$ for a variable $v$)
\begin{align*}
\partial_{p_{1}}w_{1} &  =zt(w_{1}+\frac{t}{2}\partial_{t}w_{1})+\frac{z^{2}%
t}{2}\partial_{z}w_{1},~\partial_{p_{2}}w_{1}=\frac{t}{z}(w_{1}+\frac{t}%
{2}\partial_{t}w_{1})-\frac{t}{2}\partial_{z}w_{1}\\
\partial_{p_{1}}w_{2} &  =\frac{t}{z}(w_{2}+\frac{t}{2}\partial_{t}%
w_{2})-\frac{t}{2}\partial_{z}w_{2},~\partial_{p_{2}}w_{2}=zt(w_{2}+\frac
{t}{2}\partial_{t}w_{2})+\frac{z^{2}t}{2}\partial_{z}w_{2}.
\end{align*}
After some calculations involving $\frac{\partial}{\partial z}=\frac{1}%
{2}\left(  1-z^{-2}\right)  \frac{\partial}{\partial s}$ we obtain%
\begin{align*}
\left(  \partial_{p_{1}}+\partial_{p_{2}}\right)  u_{1} &  =2stu_{1}%
+st^{2}\partial_{t}u_{1}+\left(  s^{2}-1\right)  t\partial_{s}u_{1},\\
\left(  \partial_{p_{1}}-\partial_{p_{2}}\right)  u_{1} &  =\left(
3s^{2}-2\right)  tu_{2}+\left(  s^{2}-1\right)  t^{2}\partial_{t}%
u_{2}+s\left(  s^{2}-1\right)  t\partial_{s}u_{2},\\
\left(  \partial_{p_{1}}+\partial_{p_{2}}\right)  u_{2} &  =3stu_{2}%
+st^{2}\partial_{t}u_{2}+\left(  s^{2}-1\right)  t\partial_{s}u_{2},\\
\left(  \partial_{p_{1}}-\partial_{p_{2}}\right)  u_{2} &  =2tu_{1}%
+t^{2}\partial_{t}u_{1}+st\partial_{s}u_{1}.
\end{align*}
Applying%
\begin{align*}
T_{1}+T_{2}-\partial_{p_{1}}-\partial_{p_{2}} &  =N\kappa\frac{1-\left(
p_{1}\rightarrow0\right)  }{p_{1}}+N\kappa\frac{1-\left(  p_{2}\rightarrow
0\right)  }{p_{2}},\\
T_{1}-T_{2}-\partial_{p_{1}}+\partial_{p_{2}} &  =N\kappa\frac{1-\left(
p_{1}\rightarrow0\right)  }{p_{1}}-N\kappa\frac{1-\left(  p_{2}\rightarrow
0\right)  }{p_{2}}\\
&  +2\kappa\frac{\left(  p_{1}\rightarrow p_{2}\right)  +\left(
p_{2}\rightarrow p_{1}\right)  -1-\left(  p_{1}\longleftrightarrow
p_{2}\right)  }{p_{1}-p_{2}}%
\end{align*}
to $u_{1}$ and $u_{2}$ yields%
\begin{align*}
\left(  T_{1}+T_{2}-\partial_{p_{1}}-\partial_{p_{2}}\right)  u_{1} &
=2stN\kappa u_{1},\\
\left(  T_{1}-T_{2}-\partial_{p_{1}}+\partial_{p_{2}}\right)  u_{1} &
=2t^{2}\left(  N\left(  s^{2}-1\right)  +1\right)  \kappa u_{2},\\
\left(  T_{1}+T_{2}-\partial_{p_{1}}-\partial_{p_{2}}\right)  u_{2} &
=2tsN\kappa u_{2},\\
\left(  T_{1}-T_{2}-\partial_{p_{1}}+\partial_{p_{2}}\right)  u_{2} &
=2tN\kappa u_{1}.
\end{align*}
Applying these to the generating functions results in%
\begin{align}
\left(  T_{1}+T_{2}\right)  \phi_{nj} &  =-\left(  j+1\right)  \phi
_{n-1,j+1}+\left(  2N\kappa+n+j\right)  \phi_{n-1,j-1},\label{tpts}\\
\left(  T_{1}-T_{2}\right)  \phi_{n,j} &  =-\left(  2N\kappa-2\kappa
+n+j+1\right)  \psi_{n-1,j}+\left(  2N\kappa+n+j\right)  \psi_{n-1,j-2}%
,\label{tmts}\\
\left(  T_{1}+T_{2}\right)  \psi_{nj} &  =-\left(  j+1\right)  \psi
_{n-1.j+1}+\left(  2N\kappa+n+j+1\right)  \psi_{n-1,j-1},\label{tpta}\\
\left(  T_{1}-T_{2}\right)  \psi_{nj} &  =\left(  2N\kappa+n+j+1\right)
\phi_{n-1,j}.\label{tmta}%
\end{align}

We will state the expressions for the harmonic polynomials before their
derivations, however it is necessary to define two families of polynomials via
three-term relations. The motivation comes later.

\begin{definition}
For $n=0,1,2,\ldots$ define two families of monic polynomials by%
\begin{align*}
g_{0}^{o}\left(  v\right)    & =1,g_{n+1}^{o}\left(  v\right)  =\left(
v+3n+1\right)  g_{n}^{o}\left(  v\right)  -n\left(  2n-1\right)  g_{n-1}%
^{o}\left(  v\right)  ,\\
g_{0}^{e}\left(  v\right)    & =1,g_{n+1}^{e}\left(  v\right)  =\left(
v+3n+2\right)  g_{n}^{e}\left(  v\right)  -n\left(  2n+1\right)  g_{n-1}%
^{e}\left(  v\right)  .
\end{align*}

\end{definition}

The first few polynomials are%
\begin{align*}
g_{1}^{o}\left(  v\right)    & =v+1,\\
g_{2}^{o}\left(  v\right)    & =v^{2}+5v+3,\\
g_{3}^{o}\left(  v\right)    & =v^{3}+12v^{2}+32v+15
\end{align*}
and%
\begin{align*}
g_{1}^{e}\left(  v\right)    & =v+2,\\
g_{2}^{e}\left(  v\right)    & =v^{2}+7v+7,\\
g_{3}^{e}\left(  v\right)    & =v^{3}+15v^{2}+53v+36.
\end{align*}

\begin{definition}
For $m=0,1,2,\ldots$ let%
\begin{align*}
h_{2m+1}^{-}  & :=\sum_{j=0}^{m}2^{-j}\frac{g_{j}^{o}\left(  N\kappa
-\kappa+m\right)  }{\left(  N\kappa+m+2\right)  _{j}}\psi_{2m+1,2j},\\
h_{2m}^{-}  & :=\sum_{j=0}^{m-1}2^{-j}\frac{g_{j}^{e}\left(  N\kappa
-\kappa+m\right)  }{\left(  N\kappa+m+2\right)  _{j}}\psi_{2m,2j+1},\\
h_{2m+1}^{+}  & :=\sum_{j=0}^{m}2^{-j}\frac{g_{j}^{e}\left(  N\kappa
-\kappa+m+1\right)  }{\left(  N\kappa+m+2\right)  _{j}}\phi_{2m+1,2j+1}\\
h_{2m}^{+}  & :=\sum_{j=0}^{m}2^{-j}\frac{g_{j}^{o}\left(  N\kappa
-\kappa+m\right)  }{\left(  N\kappa+m+1\right)  _{j}}\phi_{2m,2j},.
\end{align*}

\end{definition}

First we show that the antisymmetric polynomials $h_{n}^{-}$ are harmonic. We
will use the last relation to produce symmetric harmonic polynomials from the
antisymmetric ones. Combining equations (\ref{tmts}), (\ref{tmta}),
(\ref{tpta}) obtain
\begin{align*}
&  \left(  \left(  T_{1}+T_{2}\right)  ^{2}+\left(  T_{1}-T_{2}\right)
^{2}\right)  \psi_{nj}\\
&  =\left(  j+1\right)  \left(  j+2\right)  \psi_{n-2,j+2}-\left(
2N\kappa+n+j+1\right)  \left(  2N\kappa-2\kappa+n+3j+1\right)  \psi_{n-2,j}\\
&  +2\left(  2N\kappa+n+j+1\right)  \left(  2N\kappa+n-j+1\right)
\psi_{n-2,j-2}.
\end{align*}
Suppose $h_{n}^{-}=\sum_{j=0}^{\left\lfloor \left(  n-1\right)
/2\right\rfloor }c_{n-1-2j}\psi_{n,n-1-2j}$ is harmonic then the coefficient
of $\psi_{n-2,n-2j-1}$ in $2\Delta_{\kappa}h_{n}^{-}$ is
\begin{align*}
0 &  =8\left(  N\kappa+n-j+1\right)  \left(  N\kappa+n-j\right)  c_{n+1-2j}\\
&  -4\left(  N\kappa-\kappa+2n-1-3j\right)  \left(  N\kappa+n-j\right)
c_{n-1-2j}\\
&  +\left(  n-2j-2\right)  \left(  n-2j-1\right)  c_{n-3-2j}%
\end{align*}
The range of $j$ is derived from the inequality $0\leq n-2j-1\leq n-3$, that
is $1\leq j\leq\frac{n-1}{2}$. Two sets of formulae arise depending on the
parity of $n$. The equations are considered as recurrences.

Suppose $n=2m+1$ then the starting point is for $j=m$%
\[
8\left(  N\kappa+m+2\right)  \left(  N\kappa+m+1\right)  c_{2}-4\left(
N\kappa-\kappa+m+1\right)  \left(  N\kappa+m+1\right)  c_{0}=0,
\]
thus%
\[
c_{2}=\frac{1}{2}\frac{N\kappa-\kappa+m+1}{N\kappa+m+2}c_{0}.
\]
Set $j=m-i$ to obtain%
\begin{align*}
& c_{2i+2}\\
& =\frac{1}{2}\frac{\left(  N\kappa-\kappa+m+1+3i\right)  }{N\kappa
+m+2+i}c_{2i}-\frac{1}{4}\frac{\left(  2i-1\right)  i}{\left(  N\kappa
+m+2+i\right)  \left(  N\kappa+m+1+i\right)  }c_{2i-2}.
\end{align*}
To simplify the recurrences let $\gamma_{i}^{o}=2^{i}c_{2i}\left(
N\kappa+m+2\right)  _{i}/c_{0}$ then $\gamma_{0}^{o}=1$ and%
\[
\gamma_{i+1}^{o}=\left(  N\kappa-\kappa+m+1+3i\right)  \gamma_{i}^{o}-i\left(
2i-1\right)  \gamma_{i-1}^{o}.
\]
which agrees with the recurrence for $g_{i}^{o}$ with $v=N\kappa-\kappa+m$.

Thus the antisymmetric harmonic polynomial of degree $2m+1$ (normalized by
$c_{0}=1$) is%
\[
h_{2m+1}^{-}=\sum_{j=0}^{m}2^{-j}\frac{g_{j}^{o}\left(  N\kappa-\kappa
+m\right)  }{\left(  N\kappa+m+2\right)  _{j}}\psi_{2m+1,2j}.
\]
Suppose $n=2m$ then the starting point is for $j=m-1$%
\[
8\left(  N\kappa+m+2\right)  \left(  N\kappa+m+1\right)  c_{3}-4\left(
N\kappa-\kappa+m+2\right)  \left(  N\kappa+m+1\right)  c_{1}=0,
\]
so that%
\[
c_{3}=\frac{1}{2}\frac{N\kappa-\kappa+m+2}{N\kappa+m+2}c_{1}.
\]
Set $j=m-1-i$ to obtain%
\begin{align*}
& c_{2i+3}\\
& =\frac{1}{2}\frac{\left(  N\kappa-\kappa+m+2+3i\right)  }{N\kappa
+m+2+i}c_{2i+1}-\frac{1}{4}\frac{\left(  2i+1\right)  i}{\left(
N\kappa+m+2+i\right)  \left(  N\kappa+m+1+i\right)  }c_{2i-1}.
\end{align*}
Similarly to the previous calculation let $\gamma_{i}^{e}:=2^{i}%
c_{2i+1}\left(  N\kappa+m+2\right)  _{i}/c_{1}$ then $\gamma_{o}^{e}=1$ and%
\[
\gamma_{i+1}^{e}=\left(  N\kappa-\kappa+m+2+3i\right)  \gamma_{i}^{e}-i\left(
2i+1\right)  \gamma_{i-1}^{e}.
\]
This agrees with the recurrence for $g_{i}^{e}$ with $v=N\kappa-\kappa+m$.
Thus the antisymmetric harmonic polynomial of degree $2m$ (normalized by
$c_{1}=1$) is%
\[
h_{2m}^{-}=\sum_{j=0}^{m-1}2^{-j}\frac{g_{j}^{e}\left(  N\kappa-\kappa
+m\right)  }{\left(  N\kappa+m+2\right)  _{j}}\psi_{2m,2j+1}.
\]
Applying $\left(  T_{1}-T_{2}\right)  $ to a harmonic polynomial clearly
produces another harmonic polynomial; thus by equation (\ref{tmta})
\begin{align*}
&  \frac{1}{2\left(  N\kappa+m+1\right)  }\left(  T_{1}-T_{2}\right)
h_{2m+1}^{-}\\
&  =\sum_{j=0}^{m}2^{-j}\frac{g_{j}^{o}\left(  N\kappa-\kappa+m\right)
}{\left(  N\kappa+m+1\right)  _{j+1}}\left(  N\kappa+m+j+1\right)
\phi_{2m,2j}\\
&  =\sum_{j=0}^{m}2^{-j}\frac{g_{j}^{o}\left(  N\kappa-\kappa+m\right)
}{\left(  N\kappa+m+1\right)  _{j}}\phi_{2m,2j}=h_{2m}^{+}%
\end{align*}
and
\begin{align*}
&  \frac{1}{2\left(  N\kappa+m+1\right)  }\left(  T_{1}-T_{2}\right)
h_{2m}^{-}\\
&  =\sum_{j=0}^{m-1}2^{-j}\frac{g_{j}^{e}\left(  N\kappa-\kappa+m\right)
}{\left(  N\kappa+m+1\right)  _{j+1}}\left(  N\kappa+m+j+1\right)
\phi_{2m-1,2j+1}\\
&  =\sum_{j=0}^{m-1}2^{-j}\frac{g_{j}^{e}\left(  N\kappa-\kappa+m\right)
}{\left(  N\kappa+m+1\right)  _{j}}\phi_{2m-1,2j+1}=h_{2m-1}^{+}.
\end{align*}
We have proven:

\begin{proposition}
The polynomials $h_{n}^{+}$ and $h_{n}^{-}$ are harmonic.
\end{proposition}

For use in the sequel we find expressions for $T_{1}\pm T_{2}$ applied to
$h_{n}^{+}$ and $h_{n}^{-}$. 

\begin{proposition}
\label{TThaction}The actions of $T_{1}\pm T_{2}$ on the antisymmetric
polynomials $h_{n}^{-}$ are%
\begin{align*}
\left(  T_{1}-T_{2}\right)  h_{2m+1}^{-} &  =2\left(  N\kappa+m+1\right)
h_{2m}^{+},\\
\left(  T_{1}-T_{2}\right)  h_{2m}^{-} &  =2\left(  N\kappa+m+1\right)
h_{2m-1}^{+},\\
\left(  T_{1}+T_{2}\right)  h_{2m+1}^{-} &  =\left(  N\kappa-\kappa+m\right)
h_{2m}^{-},\\
\left(  T_{1}+T_{2}\right)  h_{2m}^{-} &  =2\left(  N\kappa+m+1\right)
h_{2m-1}^{-},
\end{align*}
and the actions on the symmetric polynomials $h_{n}^{+}$ are
\begin{align*}
\left(  T_{1}-T_{2}\right)  h_{2m+1}^{+} &  =-\left(  N\kappa-\kappa+m\right)
h_{2m}^{-},\\
\left(  T_{1}-T_{2}\right)  h_{2m}^{+} &  =-\left(  N\kappa-\kappa+m\right)
h_{2m-1}^{-},\\
\left(  T_{1}+T_{2}\right)  h_{2m+1}^{+} &  =2\left(  N\kappa+m+1\right)
h_{2m}^{+},\\
\left(  T_{1}+T_{2}\right)  h_{2m}^{+} &  =\left(  N\kappa-\kappa+m\right)
h_{2m-1}^{+}.
\end{align*}

\end{proposition}

Since the resulting polynomials are harmonic it suffices to consider just one
term in their expansions. The coefficients of the lowest index term (
$\phi_{2m-1,1}$, $\phi_{2m,0}$, $\psi_{2m-1,0}$, $\psi_{2m,1}$ for
$h_{2m-1}^{+},h_{2m}^{+},h_{2m-1}^{-},h_{2m}^{-}$ respectively) on the right
sides arise from at most two terms on the left. The details are in Section 5.

\section{Inner Products and Structure Constants}

Let $\mu$ denote the Gaussian measure $\left(  2\pi\right)  ^{-N/2}%
e^{-\left\vert x\right\vert ^{2}/2}\mathrm{d}x$ on $\mathbb{R}^{N}$, (where
$dx$ is the Lebesgue measure), and let $m$ denote the normalized surface
measure on $S_{N-1}:=\left\{  x\in\mathbb{R}^{N}:\left\vert x\right\vert
=1\right\}  $. The weight function $w_{\kappa}\left(  x\right)  :=\prod
\limits_{1\leq i<j\leq N}\left\vert x_{i}-x_{j}\right\vert ^{2\kappa}$.The
constants $c_{\kappa}$ and $c_{k}^{^{\prime}}$ are defined by $c_{\kappa}%
\int_{\mathbb{R}^{N}}w_{\kappa}\mathrm{d}\mu=1$ and $c_{\kappa}^{^{\prime}%
}\int_{S_{N-1}}w_{\kappa}\mathrm{d}m=1$. It is known (Macdonald-Mehta-Selberg
integral) that $c_{\kappa}=\prod\limits_{j=2}^{N}\left(  \frac{\Gamma\left(
\kappa+1\right)  }{\Gamma\left(  j\kappa+1\right)  }\right)  $. There are
three inner products for polynomials associated with $\Delta_{\kappa}$. For
polynomials $f,g$ define

\begin{enumerate}
\item $\left\langle f,g\right\rangle _{\kappa}:=f\left(  T_{1},\ldots
,T_{N}\right)  g\left(  x\right)  |_{x=0}$ (evaluated at $x=0$);

\item $\left\langle f,g\right\rangle _{G}:=c_{\kappa}\int_{\mathbb{R}^{N}%
}fgw_{\kappa}\mathrm{d}\mu$, the Gaussian inner product;

\item $\left\langle f,g\right\rangle _{S}:=c_{\kappa}^{\prime}\int_{S_{N-1}%
}fgw_{\kappa}\mathrm{d}m.$
\end{enumerate}

The details can be found in \cite[Ch. 7.2]{DX}. There are important relations
among them: $\left\langle f,g\right\rangle _{\kappa}=\left\langle
e^{-\Delta_{\kappa}/2}f,e^{-\Delta_{\kappa}/2}g\right\rangle _{G}$ (note that
the series $\sum_{j\geq0}\frac{1}{j!}\left(  -\frac{\Delta_{\kappa}}%
{2}\right)  ^{j}f$ terminates for any polynomial $f$) and if $f$ is
homogeneous of degree $2n$ then%
\[
\int_{\mathbb{R}^{N}}fw_{\kappa}\mathrm{d}\mu=2^{n+N\left(  N-1\right)
\kappa/2}\frac{\Gamma\left(  \frac{N}{2}\left(  \left(  N-1\right)
\kappa+1\right)  +n\right)  }{\Gamma\left(  \frac{N}{2}\right)  }\int
_{S_{N-1}}fw_{\kappa}\mathrm{d}m.
\]
Specialized to $f=1$ this shows that%
\[
c_{\kappa}^{^{\prime}}=2^{N\left(  N-1\right)  \kappa/2}\frac{\Gamma\left(
\frac{N}{2}\left(  \left(  N-1\right)  \kappa+1\right)  \right)  }%
{\Gamma\left(  \frac{N}{2}\right)  }c_{\kappa}%
\]
and thus%
\[
c_{\kappa}\int_{\mathbb{R}^{N}}fw_{\kappa}\mathrm{d}\mu=2^{N\left(
N-1\right)  \kappa/2}\left(  \frac{N}{2}\left(  \left(  N-1\right)
\kappa+1\right)  \right)  _{n}~c_{\kappa}^{^{\prime}}\int_{S_{N-1}}fw_{\kappa
}\mathrm{d}m.
\]
As a consequence if $f$ and $g$ are harmonic and homogeneous of degrees $m,n$
respectively then%
\begin{equation}
\left\langle f,g\right\rangle _{\kappa}=\left\langle f,g\right\rangle
_{G}=2^{n}\left(  \frac{N}{2}\left(  \left(  N-1\right)  \kappa+1\right)
\right)  _{n}~\delta_{mn}\left\langle f,g\right\rangle _{S}.\label{inproGS}%
\end{equation}
It is a fundamental result that $\deg f\neq\deg g$ implies $\left\langle
f,g\right\rangle _{S}=0$.

To find $\left\langle f,f\right\rangle _{\kappa}$ for the harmonic polynomials
$h_{n}^{+},h_{n}^{-}$ we will need the values of $\phi_{nj}$ and $\psi_{nj}$
at $x=\left(  x_{1},x_{2},0,\ldots,0\right)  $. In terms of the generating
functions%
\[
\Psi^{-1}\left(  1-rp_{1}\right)  ^{-1}=\Psi^{-1}\sum_{n=0}^{\infty}p_{1}%
^{n}r^{n}=\left(  1-rx_{1}\right)  ^{-\kappa-1}\left(  1-rx_{2}\right)
^{-\kappa},
\]
thus
\begin{align*}
w_{1}\left(  x\right)   &  :=(1-ztx_{1})^{-1-\kappa}\left(  1-ztx_{2}\right)
^{-\kappa}(1-z^{-1}tx_{2})^{-1-\kappa}\left(  1-z^{-1}tx_{1}\right)
^{-\kappa},\\
w_{2}\left(  x\right)   &  :=(1-ztx_{2})^{-1-\kappa}\left(  1-ztx_{1}\right)
^{-\kappa}(1-z^{-1}tx_{1})^{-1-\kappa}\left(  1-z^{-1}tx_{2}\right)
^{-\kappa},
\end{align*}
and%
\begin{align*}
u_{1}\left(  x\right)   &  =\frac{1}{2}\left(  w_{1}\left(  x\right)
+w_{2}\left(  x\right)  \right)  \\
&  =\frac{\left(  1-ztx_{2}\right)  \left(  1-z^{-1}tx_{1}\right)  +\left(
1-ztx_{1}\right)  \left(  1-z^{-1}tx_{2}\right)  }{2\left\{  \left(
1-2stx_{1}+x_{1}^{2}t^{2}\right)  \left(  1-2stx_{2}+x_{2}^{2}t^{2}\right)
\right\}  ^{\kappa+1}}\\
&  =\frac{1-\left(  x_{1}+x_{2}\right)  st+x_{1}x_{2}t^{2}}{\left\{  \left(
1-2stx_{1}+x_{1}^{2}t^{2}\right)  \left(  1-2stx_{2}+x_{2}^{2}t^{2}\right)
\right\}  ^{\kappa+1}},\\
u_{2}\left(  x\right)   &  =\left(  z-\frac{1}{z}\right)  ^{-1}\left(
w_{1}\left(  x\right)  -w_{2}\left(  x\right)  \right)  \\
&  =\frac{\left(  x_{1}-x_{2}\right)  t}{\left\{  \left(  1-2stx_{1}+x_{1}%
^{2}t^{2}\right)  \left(  1-2stx_{2}+x_{2}^{2}t^{2}\right)  \right\}
^{\kappa+1}},
\end{align*}
because $\left(  z-z^{-1}\right)  ^{-1}\left\{  \left(  1-ztx_{2}\right)
\left(  1-z^{-1}tx_{1}\right)  -\left(  1-ztx_{1}\right)  \left(
1-z^{-1}tx_{2}\right)  \right\}  =\left(  x_{1}-x_{2}\right)  t$. Thus
$\phi_{nj}\left(  x\right)  =\mathrm{coef}\left(  u_{1}\left(  x\right)
,t^{n}s^{j}\right)  $ and $\psi_{nj}\left(  x\right)  =\mathrm{coef}\left(
u_{2}\left(  x\right)  ,t^{n}s^{j}\right)  .$

By the $\left(  1,2\right)  $ symmetry (both $w_{\kappa}$and $h_{n}^{+}$ are
invariant and $h_{n}^{-}$ changes sign) the inner products $\left\langle
h_{n}^{+},h_{n}^{-}\right\rangle =0$. Next we compute the pairing
$\left\langle f,f\right\rangle _{\kappa}$ for the harmonic polynomials. Since
they are annihilated by $T_{j}$ for $j>2$ these values are given by $f\left(
T_{1},T_{2},0,\ldots\right)  f$. We use the harmonicity of $f,$that is,
$\left(  T_{1}-T_{2}\right)  ^{2}f=-\left(  T_{1}+T_{2}\right)  ^{2}f.$ The
same relation holds when $f$ is replaced by $q\left(  T_{1},T_{2}\right)  f$
for any polynomial $q$. Suppose $\deg f=n$ and express%
\[
f\left(  x_{1},x_{2},0\ldots\right)  =\sum_{j=0}^{n}c_{j}\left(  x_{1}%
+x_{2}\right)  ^{n-j}\left(  x_{1}-x_{2}\right)  ^{j}.
\]
If $\left(  1,2\right)  f=f$ then $c_{j}=0$ for odd $j$ and%
\begin{align*}
f\left(  T_{1},T_{2},0\ldots\right)  f &  =\sum_{j=0}^{\left\lfloor
n/2\right\rfloor }c_{2j}\left(  T_{1}+T_{2}\right)  ^{n-2j}\left(  T_{1}%
-T_{2}\right)  ^{2j}f\\
&  =\sum_{j=0}^{\left\lfloor n/2\right\rfloor }c_{2j}\left(  -1\right)
^{j}\left(  T_{1}+T_{2}\right)  ^{n}f.
\end{align*}
Set $x_{1}=1+\mathrm{i},x_{2}=1-\mathrm{i}$ then%
\[
f\left(  x\right)  =\sum_{j=0}^{\left\lfloor n/2\right\rfloor }c_{2j}%
2^{n-2j}\left(  2\mathrm{i}\right)  ^{2j}=2^{n}\sum_{j=0}^{\left\lfloor
n/2\right\rfloor }c_{2j}\left(  -1\right)  ^{j};
\]
and thus%
\[
f\left(  T_{1},T_{2},0\ldots\right)  f=2^{-n}f\left(  1+\mathrm{i}%
,1-\mathrm{i},0\mathrm{\ldots}\right)  \left(  T_{1}+T_{2}\right)  ^{n}f.
\]
Proceeding similarly for $\left(  1,2\right)  f=-f$ where $c_{j}=0$ for even
$j$ we obtain%
\begin{align*}
f\left(  T_{1},T_{2},0\ldots\right)  f &  =\sum_{j=0}^{\left\lfloor \left(
n-1\right)  /2\right\rfloor }c_{2j+1}\left(  T_{1}+T_{2}\right)
^{n-1-2j}\left(  T_{1}-T_{2}\right)  ^{2j+1}f\\
&  =\sum_{j=0}^{\left\lfloor \left(  n-1\right)  /2\right\rfloor }%
c_{2j+1}\left(  -1\right)  ^{j}\left(  T_{1}-T_{2}\right)  \left(  T_{1}%
+T_{2}\right)  ^{n-1}f,
\end{align*}
and%
\[
f\left(  1+\mathrm{i},1-\mathrm{i},0\mathrm{\ldots}\right)  =\sum
_{j=0}^{\left\lfloor \left(  n-1\right)  /2\right\rfloor }c_{2j+1}%
2^{n-1-2j}\left(  2\mathrm{i}\right)  ^{2j+1}=\mathrm{i2}^{n}\sum
_{j=0}^{\left\lfloor \left(  n-1\right)  /2\right\rfloor }c_{2j+1}\left(
-1\right)  ^{j}.
\]
Thus%
\[
f\left(  T_{1},T_{2},0\ldots\right)  f=-\mathrm{i}2^{-n}f\left(
1+\mathrm{i},1-\mathrm{i},0\mathrm{\ldots}\right)  \left(  T_{1}+T_{2}\right)
^{n-1}\left(  T_{1}-T_{2}\right)  f.
\]

First the symmetric case (by Proposition \ref{TThaction}):%
\[
\left(  T_{1}+T_{2}\right)  ^{2}h_{2m}^{+}=\left(  T_{1}+T_{2}\right)  \left(
N\kappa-\kappa+m\right)  h_{2m-1}^{+}=2\left(  N\kappa-\kappa+m\right)
\left(  N\kappa+m\right)  h_{2m-2}^{+},
\]
and it follows by induction that%
\begin{align*}
\left(  T_{1}+T_{2}\right)  ^{2m}h_{2m}^{+} &  =2^{m}\left(  N\kappa
-\kappa+1\right)  _{m}\left(  N\kappa+1\right)  _{m},\\
\left(  T_{1}+T_{2}\right)  ^{2m+1}h_{2m+1}^{+} &  =2\left(  N\kappa
+m+1\right)  \left(  T_{1}+T_{2}\right)  ^{2m}h_{2m}^{+}\\
&  =2^{m+1}\left(  N\kappa-\kappa+1\right)  _{m}\left(  N\kappa+1\right)
_{m+1}.
\end{align*}
For the antisymmetric case%
\begin{align*}
& \left(  T_{1}+T_{2}\right)  ^{2m-1}\left(  T_{1}-T_{2}\right)  h_{2m}^{-}\\
& =2\left(  N\kappa+m+1\right)  \left(  T_{1}+T_{2}\right)  ^{2m-1}%
h_{2m-1}^{+}\\
& =2^{m+1}\left(  N\kappa+m+1\right)  \left(  N\kappa-\kappa+1\right)
_{m-1}\left(  N\kappa+1\right)  _{m}\\
& =2^{m+1}\left(  N\kappa-\kappa+1\right)  _{m-1}\left(  N\kappa+1\right)
_{m+1},
\end{align*}
and%
\begin{align*}
& \left(  T_{1}+T_{2}\right)  ^{2m}\left(  T_{1}-T_{2}\right)  h_{2m+1}^{-}\\
& =2\left(  N\kappa+m+1\right)  \left(  T_{1}+T_{2}\right)  ^{2m}h_{2m}^{+}\\
& =2^{m+1}\left(  N\kappa+m+1\right)  \left(  N\kappa-\kappa+1\right)
_{m}\left(  N\kappa+1\right)  _{m}\\
& =2^{m+1}\left(  N\kappa-\kappa+1\right)  _{m}\left(  N\kappa+1\right)
_{m+1}.
\end{align*}
The values $\phi_{nj}\left(  1+\mathrm{i},1-\mathrm{i},0\ldots\right)  $ and
$\psi_{nj}\left(  1+\mathrm{i},1-\mathrm{i},0\ldots\right)  $ are found by
computing the generating functions:%
\[
u_{1}\left(  1+\mathrm{i},1-\mathrm{i},0\ldots\right)  =\frac{1-2st+2t^{2}%
}{\left(  1-4st+8s^{2}t^{2}-8st^{3}+4t^{4}\right)  ^{\kappa+1}}%
\]
(the term in the denominator is $\left(  1-2st+2t^{2}\right)  ^{2}-4\left(
1-s^{2}\right)  t^{2}$) and%
\[
u_{2}\left(  1+\mathrm{i},1-\mathrm{i},0\ldots\right)  =\frac{2\mathrm{i}%
t}{\left(  1-4st+8s^{2}t^{2}-8st^{3}+4t^{4}\right)  ^{\kappa+1}}.
\]

\begin{definition}
\label{Snjab}For $n=0,1,2,\ldots$, $0\leq j\leq\left\lfloor \frac{n}%
{2}\right\rfloor $ and parameters $\alpha,\beta$ let%
\begin{multline*}
S\left(  n,j;\alpha,\beta\right)  \\
:=\sum_{\ell=0}^{\left\lfloor n/2\right\rfloor }\sum_{i=\max\left(
0,\ell+j-\left\lfloor n/2\right\rfloor \right)  }^{\min\left(  \ell,j\right)
}\frac{\left(  \alpha+1\right)  _{\ell}\left(  2\alpha+\beta+2\ell\right)
_{n-2\ell-j+i}}{i!\left(  \ell-i\right)  !\left(  j-i\right)  !\left(
n-2\ell-2j+2i\right)  !}\left(  -1\right)  ^{\ell+j}2^{n-j+i}.
\end{multline*}

\end{definition}

\begin{proposition}
\label{val1pI}For $0\leq j\leq\left\lfloor \frac{n}{2}\right\rfloor $%
\begin{align*}
\phi_{n,n-2j}\left(  1+\mathrm{i},1-\mathrm{i},0,\ldots\right)   &  =S\left(
n,j;\kappa,1\right)  ,\\
\psi_{n+1,n-2j}\left(  1+\mathrm{i},1-\mathrm{i},0,\ldots\right)   &
=2\mathrm{i}S\left(  n,j;\kappa,2\right)  .
\end{align*}

\end{proposition}

The proof is in Proposition \ref{genfunS}.

Thus%
\begin{align*}
\left\langle h_{2m}^{+},h_{2m}^{+}\right\rangle _{\kappa}  & =2^{-2m}%
h_{2m}^{+}\left(  1+\mathrm{i},1-\mathrm{i},0\mathrm{\ldots}\right)  \left(
T_{1}+T_{2}\right)  ^{2m}h_{2m}^{+}\\
& =2^{-m}\sum_{j=0}^{m}2^{-j}\frac{g_{j}^{o}\left(  N\kappa-\kappa+m\right)
}{\left(  N\kappa+m+1\right)  _{j}}S\left(  2m,m-j;\kappa,1\right)  \\
& \times\left(  N\kappa-\kappa+1\right)  _{m}\left(  N\kappa+1\right)  _{m},\\
\left\langle h_{2m+1}^{+},h_{2m+1}^{+}\right\rangle _{\kappa}  &
=2^{-2m-1}h_{2m+1}^{+}\left(  1+\mathrm{i},1-\mathrm{i},0\mathrm{\ldots
}\right)  \left(  T_{1}+T_{2}\right)  ^{2m+1}h_{2m+1}^{+}\\
& =2^{-m}\sum_{j=0}^{m}2^{-j}\frac{g_{j}^{e}\left(  N\kappa-\kappa+m+1\right)
}{\left(  N\kappa+m+2\right)  _{j}}S\left(  2m+1,m-j;\kappa,1\right)  \\
& \times\left(  N\kappa-\kappa+1\right)  _{m}\left(  N\kappa+1\right)  _{m+1},
\end{align*}
and%
\begin{align*}
\left\langle h_{2m}^{-},h_{2m}^{-}\right\rangle _{\kappa}  & =-\mathrm{i}%
2^{-2m}h_{2m}^{-}\left(  1+\mathrm{i},1-\mathrm{i},0\mathrm{\ldots}\right)
\left(  T_{1}+T_{2}\right)  ^{2m-1}\left(  T_{1}-T_{2}\right)  h_{2m}^{-}\\
& =2^{-m+2}\sum_{j=0}^{m-1}2^{-j}\frac{g_{j}^{e}\left(  N\kappa-\kappa
+m\right)  }{\left(  N\kappa+m+2\right)  _{j}}S\left(  2m,m-j-1;\kappa
,2\right)  \\
& \times\left(  N\kappa-\kappa+1\right)  _{m-1}\left(  N\kappa+1\right)
_{m+1},\\
\left\langle h_{2m+1}^{-},h_{2m+1}^{-}\right\rangle _{\kappa}  &
=-\mathrm{i}2^{-2m}h_{2m+1}^{-}\left(  1+\mathrm{i},1-\mathrm{i}%
,0\mathrm{\ldots}\right)  \left(  T_{1}+T_{2}\right)  ^{2m}\left(  T_{1}%
-T_{2}\right)  h_{2m+1}^{-}\\
& =2^{2-m}\sum_{j=0}^{m}2^{-j}\frac{g_{j}^{o}\left(  N\kappa-\kappa+m\right)
}{\left(  N\kappa+m+2\right)  _{j}}S\left(  2m+1,m-j;\kappa,2\right)  \\
& \times\left(  N\kappa-\kappa+1\right)  _{m}\left(  N\kappa+1\right)  _{m+1}.
\end{align*}
The values of $\left\langle h_{n}^{+},h_{n}^{+}\right\rangle _{S}$ and
$\left\langle h_{n}^{-},h_{n}^{-}\right\rangle _{S}$ can now be found by
equation (\ref{inproGS}). The expressions are complicated; due to the fact
that sign-changes are not in the symmetry group.

\section{The Dirac Operator and Monogenic Polynomials}

We use the Clifford algebra $C\ell_{N}$ over $\mathbb{R}$ generated by
$\left\{  \mathbf{e}_{1},\mathbf{e}_{2},\mathbf{e}_{3},\ldots,\mathbf{e}%
_{N}\right\}  $ with relations $\mathbf{e}_{i}^{2}=-1$ (that is, negative
signature) and $\mathbf{e}_{i}\mathbf{e}_{j}=-\mathbf{e}_{j}\mathbf{e}_{i}$
for $i\neq j$. The type-$A$ Dirac operator acting on polynomials in
$x\in\mathbb{R}^{N}$ with coefficients in $C\ell_{N}$ is defined by%
\[
\boldsymbol{D}f:=\sum_{i=1}^{N}\mathbf{e}_{i}T_{i}f;
\]
this implies $\boldsymbol{D}^{2}=-\sum_{i=1}^{N}T_{i}^{2}=-\Delta_{\kappa}$. A
polynomial $f$ is said to be monogenic if $\boldsymbol{D}f=0$. The situation
where the underlying symmetry group is $\mathbb{Z}_{2}^{N}$ has been
investigated by De Bie, Genest and Vinet \cite{DGV1},\cite{DGV2}. The planar
harmonic polynomials found in the previous sections can be used to construct
monogenic polynomials. They are of the form $f_{n}=h_{n}^{+}+\varepsilon
h_{n}^{-}$ with $\varepsilon\in C\ell_{N}$. By construction $T_{i}f=0$ for all
$i>2$. To fit with the formulae in Proposition \ref{TThaction} write
$\mathbf{e}_{1}T_{1}+\mathbf{e}_{2}T_{2}=\frac{1}{2}\left(  \mathbf{e}%
_{1}+\mathbf{e}_{2}\right)  \left(  T_{1}+T_{2}\right)  +\frac{1}{2}\left(
\mathbf{e}_{1}-\mathbf{e}_{2}\right)  \left(  T_{1}-T_{2}\right)  $. Even and
odd $n$ are handled separately.%
\begin{align*}
&  \left(  \mathbf{e}_{1}T_{1}+\mathbf{e}_{2}T_{2}\right)  \left(
h_{2m+1}^{+}+\varepsilon h_{2m+1}^{-}\right)  \\
&  =\left(  \mathbf{e}_{1}+\mathbf{e}_{2}\right)  \left(  N\kappa+m+1\right)
h_{2m}^{+}-\frac{1}{2}\left(  \mathbf{e}_{1}-\mathbf{e}_{2}\right)  \left(
N\kappa-\kappa+m\right)  h_{2m}^{-}\\
&  +\left(  \mathbf{e}_{1}-\mathbf{e}_{2}\right)  \varepsilon\left(
N\kappa+m+1\right)  h_{2m}^{+}+\frac{1}{2}\left(  \mathbf{e}_{1}%
+\mathbf{e}_{2}\right)  \varepsilon\left(  N\kappa-\kappa+m\right)  h_{2m}%
^{-};
\end{align*}
the coefficients of $\left(  N\kappa+m+1\right)  h_{2m}^{+}$ and $\frac{1}%
{2}\left(  N\kappa-\kappa+m\right)  h_{2m}^{-}$ are $\left(  \mathbf{e}%
_{1}+\mathbf{e}_{2}\right)  +\left(  \mathbf{e}_{1}-\mathbf{e}_{2}\right)
\varepsilon$ and $-\left(  \mathbf{e}_{1}-\mathbf{e}_{2}\right)  +\left(
\mathbf{e}_{1}+\mathbf{e}_{2}\right)  \varepsilon$ respectively. Both of these
vanish for $\varepsilon=\mathbf{e}_{1}\mathbf{e}_{2}.$Thus $\boldsymbol{D}%
\left(  h_{2m+1}^{+}+\mathbf{e}_{1}\mathbf{e}_{2}h_{2m+1}^{-}\right)  =0$.
Since $\boldsymbol{D}$ commutes with $T_{1}+T_{2}$ the polynomial $\left(
T_{1}+T_{2}\right)  \left(  h_{2m+1}^{+}+\mathbf{e}_{1}\mathbf{e}_{2}%
h_{2m+1}^{-}\right)  $ is also monogenic, and $\left(  T_{1}+T_{2}\right)
\left(  h_{2m+1}^{+}+\mathbf{e}_{1}\mathbf{e}_{2}h_{2m+1}^{-}\right)
=2\left(  N\kappa+m+1\right)  h_{2m}^{+}+\mathbf{e}_{1}\mathbf{e}_{2}\left(
N\kappa-\kappa+m\right)  h_{2m}^{-}$. This proves
\begin{align*}
\boldsymbol{D}\left(  h_{2m+1}^{+}+\mathbf{e}_{1}\mathbf{e}_{2}h_{2m+1}%
^{-}\right)   &  =0,\\
\boldsymbol{D}\left(  h_{2m}^{+}+\mathbf{e}_{1}\mathbf{e}_{2}\frac
{N\kappa-\kappa+m}{2\left(  N\kappa+m+1\right)  }h_{2m}^{-}\right)   &  =0.
\end{align*}

\section{Derivations of Various Formulae}

This section contains the derivations of some of the formulae appearing in the
paper. The formulae for $\phi_{n,j}$ and $\psi_{n,j}$ are found by means of
the Chebyshev polynomials $T_{k}$ and $U_{k}$.%
\begin{align*}
u_{1}  & =\frac{1}{2}\sum_{k=0}^{\infty}\sum_{m=0}^{\infty}t^{k+m}p_{1}%
^{k}p_{2}^{m}\left(  z^{k-m}+z^{m-k}\right)  \\
& =\frac{1}{2}\sum_{n=0}^{\infty}t^{n}\sum_{m=0}^{n}p_{1}^{n-m}p_{2}%
^{m}\left(  z^{n-2m}+z^{2m-n}\right)  \\
& =\sum_{n=0}^{\infty}t^{n}\sum_{m=0}^{n}p_{1}^{n-m}p_{2}^{m}\cos\left(
\left(  n-2m\right)  \theta\right)  \\
& =\sum_{n=0}^{\infty}t^{n}\sum_{m=0}^{n}p_{1}^{n-m}p_{2}^{m}T_{\left\vert
n-2m\right\vert }\left(  s\right)  ,
\end{align*}
where $z$ is replaced by $e^{\mathrm{i}\theta}$ and thus $s=\cos\theta$. The
last inner sum can be written as $\sum_{m=0}^{\left\lfloor n/2\right\rfloor
}\varepsilon_{n,m}\left(  p_{1}^{n-m}p_{2}^{m}+p_{1}^{m}p_{2}^{n-m}\right)
T_{n-2m}\left(  s\right)  $ where $\varepsilon_{n,m}=1$ except $\varepsilon
_{2m,m}=\frac{1}{2}$. Then use the expansion $T_{k}\left(  s\right)
=\sum\limits_{j=0}^{\left\lfloor k/2\right\rfloor }\dfrac{\left(  -k\right)
_{2j}}{j!\left(  1-k\right)  _{j}}2^{k-1-2j}s^{k-2j}$ for $k\geq1$ and extract
the coefficient of $s^{n-2j}$ to determine $\phi_{n,n-2j}$. Applying the same
technique to $u_{2}$ we obtain%
\begin{align*}
u_{2}  & =\frac{1}{z-z^{-1}}\sum_{n=0}^{\infty}t^{n}\sum_{m=0}^{n}p_{1}%
^{n-m}p_{2}^{m}\left(  z^{n-2m}-z^{2m-n}\right)  \\
& =\sum_{n=0}^{\infty}t^{n}\sum_{m=0}^{n}p_{1}^{n-m}p_{2}^{m}\frac{\sin\left(
\left(  n-2m\right)  \theta\right)  }{\sin\theta}\\
& =\sum_{n=0}^{\infty}t^{n}\sum_{m=0}^{\left\lfloor n/2\right\rfloor }\left(
p_{1}^{n-m}p_{2}^{m}-p_{1}^{m}p_{2}^{n-m}\right)  \frac{\sin\left(  \left(
n-2m\right)  \theta\right)  }{\sin\theta}\\
& =\sum_{n=0}^{\infty}t^{n}\sum_{m=0}^{\left\lfloor n/2\right\rfloor }\left(
p_{1}^{n-m}p_{2}^{m}-p_{1}^{m}p_{2}^{n-m}\right)  U_{n-1-2m}\left(  s\right)
.
\end{align*}
Then extract the coefficient of $s^{n-1-2j}$ by means of the expansion
$U_{k}\left(  s\right)  =\sum\limits_{j=0}^{\left\lfloor k/2\right\rfloor
}\dfrac{\left(  -k\right)  _{2j}}{j!\left(  -k\right)  _{j}}2^{k-2j}s^{k-2j}$
for $k>0$ to find $\psi_{n,n-1-2j}$.

\begin{proof}
(of Proposition \ref{TThaction}). The formulae for $\left(  T_{1}%
-T_{2}\right)  h_{n}^{-}$ have already been proven. For $\left(  T_{1}%
+T_{2}\right)  h_{n}^{-}$ substitute $n=2m$ and $j=1$ in (\ref{tpta}) to
obtain $\mathrm{coef}\left(  \left(  T_{1}+T_{2}\right)  h_{2m}^{-}%
,\psi_{2m-1,0}\right)  =\left(  2N\kappa+2m+2\right)  $ and thus $\left(
T_{1}+T_{2}\right)  h_{2m}^{-}=2\left(  N\kappa+m+1\right)  h_{2m-1}^{-}$,
next substitute $n=2m+1$ and $j=0,2$ in (\ref{tpta}) to show that%
\begin{align*}
& \mathrm{coef}\left(  \left(  T_{1}+T_{2}\right)  h_{2m+1}^{-},\psi
_{2m,1}\right)  \\
& =-\mathrm{coef}\left(  h_{2m+1}^{-},\psi_{2m+1,0}\right)  +\left(
2N\kappa+2m+4\right)  \mathrm{coef}\left(  h_{2m+1}^{-},\psi_{2m+1,2}\right)
\\
& =-1+\left(  N\kappa+m+2\right)  \frac{g_{1}^{o}\left(  N\kappa
-\kappa+m\right)  }{N\kappa+m+2}=N\kappa-\kappa+m.
\end{align*}
For $\left(  T_{1}+T_{2}\right)  h_{n}^{+}$ substitute $n=2m+1,j=1$ in
(\ref{tpts}) to show%
\begin{align*}
\mathrm{coef}\left(  \left(  T_{1}+T_{2}\right)  h_{2m+1}^{+},\phi
_{2m,0}\right)    & =\left(  2N\kappa+2m+2\right)  \mathrm{coef}\left(
h_{2m+1}^{+},\phi_{2m+1,1}\right)  \\
& =2\left(  N\kappa+m+1\right)  ;
\end{align*}
next substitute $n=2m,j=0,2$ in (\ref{tpts}) to show that%
\begin{align*}
& \mathrm{coef}\left(  \left(  T_{1}+T_{2}\right)  h_{2m}^{=},\phi
_{2m-1,1}\right)  \\
& =-\mathrm{coef}\left(  h_{2m}^{+},\phi_{2m,0}\right)  +\left(
2N\kappa+2m+2\right)  \mathrm{coef}\left(  h_{2m}^{+},\phi_{2m,2}\right)  \\
& =-1+\left(  N\kappa+m+1\right)  \frac{g_{1}^{o}\left(  N\kappa
-\kappa+m\right)  }{N\kappa+m+1}=N\kappa-\kappa+m.
\end{align*}
For $\left(  T_{1}-T_{2}\right)  h_{n}^{+}$ substitute $n=2m+1,j=1,3$ in
(\ref{tmts}) to show%
\begin{align*}
& \mathrm{coef}\left(  \left(  T_{1}-T_{2}\right)  h_{2m+1}^{+},\psi
_{2m,1}\right)  \\
& =-\left(  2N\kappa-2\kappa+2m+3\right)  \mathrm{coef}\left(  h_{2m+1}%
^{+},\phi_{2m+1,1}\right)  \\
& +\left(  2N\kappa+2m+4\right)  \mathrm{coef}\left(  h_{2m+1}^{+}%
,\phi_{2m+1,3}\right)  \\
& =-\left(  2N\kappa-2\kappa+2m+3\right)  +\left(  N\kappa+m+2\right)
\frac{g_{1}^{e}\left(  N\kappa-\kappa+m+1\right)  }{N\kappa+m+2}\\
& =-\left(  N\kappa-\kappa+m\right)  ;
\end{align*}
next substitute $n=2m,j=0,2$ in (\ref{tmts}) to show%
\begin{align*}
& \mathrm{coef}\left(  \left(  T_{1}-T_{2}\right)  h_{2m}^{+},\psi
_{2m-1,0}\right)  \\
& =-\left(  2N\kappa-2\kappa+2m+1\right)  \mathrm{coef}\left(  h_{2m}^{+}%
,\phi_{2m,0}\right)  +\left(  2N\kappa+2m+2\right)  \mathrm{coef}\left(
h_{2m}^{+},\phi_{2m,2}\right)  \\
& \\
& =-\left(  2N\kappa-2\kappa+2m+1\right)  +\left(  N\kappa+m+1\right)
\frac{g_{1}^{o}\left(  N\kappa-\kappa+m\right)  }{N\kappa+m+1}\\
& =-\left(  N\kappa-\kappa+m\right)  .
\end{align*}
This completes the proof of Proposition \ref{TThaction}.
\end{proof}

To prove Proposition \ref{val1pI} note that the expressions for $u_{1}$ and
$u_{2}/\left(  2\mathrm{i}t\right)  $ have the form%
\[
\left(  1-2st+2t^{2}\right)  ^{-2\kappa-\beta}\left(  1-\frac{4\left(
1-s^{2}\right)  t^{2}}{\left(  1-2st+2t^{2}\right)  ^{2}}\right)  ^{-\kappa-1}%
\]
with $\beta=1$ and $2$, respectively.

\begin{proposition}
\label{genfunS}For any $\alpha,\beta$ and $\left\vert t\right\vert <\frac
{1}{\sqrt{2}}\min\left\{  \left\vert s\pm\sqrt{s^{2}-1}\right\vert \right\}  $%
\[
\frac{\left(  1-2st+2t^{2}\right)  ^{2-\beta}}{\left(  1-4st+8s^{2}%
t^{2}-8st^{3}+4t^{4}\right)  ^{\alpha+1}}=\sum_{n=0}^{\infty}\sum
_{j=0}^{\left\lfloor n/2\right\rfloor }S\left(  n,j;\alpha,\beta\right)
t^{n}s^{n-2j}%
\]
where $S\left(  n,j;\alpha,\beta\right)  $ is given in Definition \ref{Snjab}.
\end{proposition}

\begin{proof}
Denote the left hand side by $G\left(  s,t;\alpha,\beta\right)  $. The
expansion process begins with%
\begin{align*}
& \left(  1-2st+2t^{2}\right)  ^{-2\alpha-\beta}\left(  1-\frac{4\left(
1-s^{2}\right)  t^{2}}{\left(  1-2st+2t^{2}\right)  ^{2}}\right)  ^{-\alpha
-1}\\
& =\sum_{\ell=0}^{\infty}\frac{\left(  \alpha+1\right)  _{\ell}}{\ell
!}2^{2\ell}\left(  1-s^{2}\right)  ^{\ell}t^{2\ell}\left(  1-2st+2t^{2}%
\right)  ^{-2\alpha-\beta-2\ell}.
\end{align*}
By a variant of the generating function for Gegenbauer polynomials with
$\lambda>0$%
\begin{align*}
\left(  1-2st+2t^{2}\right)  ^{-\lambda} &  =\left(  1+2t^{2}\right)
^{-\lambda}\sum_{k=0}^{\infty}\frac{\left(  \lambda\right)  _{k}}{k!}\left(
2st\right)  ^{k}\left(  1+2t^{2}\right)  ^{-k}\\
&  =\sum_{k=0}^{\infty}\sum_{m=0}^{\infty}\frac{\left(  \lambda\right)  _{k}%
}{k!}\left(  2st\right)  ^{k}\frac{\left(  \lambda+k\right)  _{m}}{m!}\left(
-2t^{2}\right)  ^{m}\\
&  =\sum_{n=0}^{\infty}t^{n}\sum_{m=0}^{\left\lfloor n/2\right\rfloor }%
\frac{\left(  \lambda\right)  _{n-m}}{\left(  n-2m\right)  !m!}\left(
-1\right)  ^{m}2^{n-m}s^{n-2m},
\end{align*}
changing the summation index $k=n-2m$. Combining the expressions results in%
\begin{align*}
& G\left(  s,t;\alpha,\beta\right)  \\
& =\sum_{\ell=0}^{\infty}\frac{\left(  \alpha+1\right)  _{\ell}}{\ell!}\left(
2t\right)  ^{2\ell}\left(  1-s^{2}\right)  ^{\ell}\sum_{k=0}^{\infty}t^{k}%
\sum_{m=0}^{\left\lfloor k/2\right\rfloor }\frac{\left(  2\alpha+\beta
+2\ell\right)  _{k-m}}{\left(  k-2m\right)  !m!}\left(  -1\right)  ^{m}%
2^{k-m}s^{k-2m}\\
& =\sum_{n=0}^{\infty}t^{n}\sum_{\ell=0}^{\left\lfloor n/2\right\rfloor }%
\frac{\left(  \alpha+1\right)  _{\ell}}{\ell!}\left(  1-s^{2}\right)  ^{\ell
}\sum_{m=0}^{\left\lfloor n/2\right\rfloor -\ell}\frac{\left(  2\alpha
+\beta+2\ell\right)  _{n-2\ell-m}}{\left(  n-2\ell-2m\right)  !m!}\left(
-1\right)  ^{m}2^{n-m}s^{n-2m-2\ell},
\end{align*}
changing the summation indices to $k=n-2\ell$. Expand $\left(  1-s^{2}\right)
^{\ell}=\sum\limits_{i=0}^{\ell}\binom{\ell}{i}\left(  -s^{2}\right)
^{\ell-i}$ and change indices replacing $m$ by $j-i$. Then%
\begin{align*}
& G\left(  s,t;\alpha,\beta\right)  =\sum_{n=0}^{\infty}t^{n}\sum
_{j=0}^{\left\lfloor n/2\right\rfloor }s^{n-2j}\\
& \times\sum_{\ell=0}^{\left\lfloor n/2\right\rfloor }\sum_{i=\max\left(
0,\ell+j-\left\lfloor n/2\right\rfloor \right)  }^{\min\left(  \ell,j\right)
}\frac{\left(  \alpha+1\right)  _{\ell}\left(  2\alpha+\beta+2\ell\right)
_{n-2\ell-j+i}}{i!\left(  \ell-i\right)  !\left(  j-i\right)  !\left(
n-2\ell-2j+2i\right)  !}\left(  -1\right)  ^{\ell+j}2^{n-j+i}.
\end{align*}
The summation limits on $i$ are derived from the bounds $0\leq i\leq\ell$,
$0\leq i\leq j$, and $n-2\ell-2j+2i\geq0$. the last bound implies $i\geq
\ell+j-\frac{n}{2}$ (if $n=2m+1$ the bound is $i\geq\ell+j-m$ and
$m=\left\lfloor n/2\right\rfloor $). The bounds for $s,t$ imply that the two
factors $\left(  1-2xst+x^{2}t^{2}\right)  $ for $x=1\pm\mathrm{i}$ do not
vanish for $\left\vert \sqrt{2}t\right\vert <\min\left\vert s\pm\sqrt{s^{2}%
-1}\right\vert $ and this is sufficient for convergence of the series (if
$s=\frac{1}{2}\left(  z+z^{-1}\right)  $ then the convergence requirement is
$\left\vert \sqrt{2}t\right\vert <\min\left(  \left\vert z\right\vert
,\left\vert z\right\vert ^{-1}\right)  $). This completes the proof for the
formula for $S\left(  n,j;\alpha,\beta\right)  $.
\end{proof}

Investigating harmonic polynomials in $p_{1},p_{2},p_{3}$ which are
$\left\langle \cdot,\cdot\right\rangle _{\kappa}$-orthogonal to the planar
polynomials might be a plausible topic for further research.

\end{document}